\documentclass[11pt]{article}
\usepackage[utf8]{inputenc} 
\usepackage[T1]{fontenc} 
\usepackage{amsmath}
\usepackage{graphicx}
\usepackage{amsfonts}
\usepackage{amssymb}
\usepackage{xcolor}
\usepackage{amsthm}
\usepackage{amscd}
\usepackage{mathrsfs}
\usepackage{float}
\usepackage{appendix}
\usepackage{subfigure}
\usepackage{glossaries}
\usepackage{stmaryrd}
\usepackage{makeidx}
\usepackage{geometry}
\usepackage{enumerate}
\usepackage[pagewise]{lineno}

\newtheorem{teo}{Theorem}[section]

\newtheorem{prop}[teo]{Proposition}

\newtheorem{lem}[teo]{Lemma}

\newtheorem{defi}[teo]{Definition}

\renewcommand{\Phi}{\varPhi}
\newcommand{\R}{{\mathbb{R}}}

\newcommand{\C}{{\mathbb{C}}}

\makeglossaries
\makeindex

\title{\LARGE \bf
On the compatibility between the adiabatic and the rotating wave
approximations in quantum control
}

\author{Nicolas Augier\footnote{CMAP, \'Ecole Polytechnique, Institut Polytechnique de Paris}\,, 
Ugo Boscain\footnote{CNRS}\;$^\S$\,, Mario Sigalotti \footnote{Inria Paris} \footnote{Sorbonne Universit\'e, Inria, Universit\'{e} de Paris, CNRS, Laboratoire Jacques-Louis Lions, Paris, France}}

\begin{document}
\maketitle

\begin{abstract}
In this paper, we discuss the compatibility between the \emph{rotating-wave} and the \emph{adiabatic} approximations for controlled quantum systems.
Although the paper focuses on applications to two-level quantum systems, the main results apply in higher dimension.
Under some suitable hypotheses on the time scales, the two approximations can be combined. As a natural consequence of this, it is possible to design control laws achieving transitions of states between two energy levels of the Hamiltonian that are robust with respect to inhomogeneities of the amplitude of the control input.
\end{abstract}

\section{INTRODUCTION}

An important issue of quantum control is to design explicit control laws for the problem of the single input bilinear Schr\"odinger equation, that is
\begin{equation}\label{single}
i\frac{d\psi}{dt}=\left(H_0+uH_1\right)\psi
\end{equation}
where $\psi$ belongs to the unit sphere in a Hilbert space $\mathcal{H}$.
$H_0$ is a self adjoint operator representing a drift term called \emph{free Hamiltonian}, $H_1$ is a self-adjoint operator representing the control coupling and $u:[0,T] \to \R$, $T>0$.
Important theoretical results of controllability have been proved with different techniques (see \cite{DAlessandro,BeauchardCoron,BCCS} and references therein).
For the problem with two or more inputs, adiabatic methods are a nowadays classical way to get an explicit expression of the controls and can be used under geometric conditions on the spectrum of the controlled Hamiltonian (see \cite{Ensemble,Bos,RouchonSarlette} and references therein).  
They rely on the adiabatic theorem and its generalisations. The adiabatic theorem states in its simplest form that under a separation condition on the energy levels of the controled Hamiltonian, the occupation probabilities of the energy levels are approximately conserved when the controls are slowly varying. 
However, these methods are effective for inputs of dimension at least 2. 
Our aim is then to extend a single-input bilinear Schr\"odinger equation into a two-inputs bilinear Schr\"odinger equation in the same spirit as the Lie-extensions obtained by Sussmann and Liu in \cite{Liu97} and \cite{Sussmann1993}, then to apply the well-known adiabatic techniques to the extended system. The first step of this procedure is well known by physicists and it is called the rotating-wave approximation (RWA, for short).
It is a decoupling approximation to get 
rid of highly oscillating terms  when the system is driven by a real control. This approximation is based on a first-order averaging procedure (see \cite{Sanders, Sussmann1993, Liu97,Bullo} for more informations about averaging of dynamical systems). 
This approximation is known to work well for a small detuning from the resonance frequency and a small amplitude.
For a review of the RWA and its limitations see \cite{fujii2017} and \cite{Guerin98,Guerin99,Irish}. In \cite{Chambrion}, the mathematical framework has been set for infinite-dimensional quantum systems, formalizing what physicists call \emph{Generalized Rabi oscillations} and showing that the RWA is valid for a large class of quantum systems.
The adiabatic and RWA involve different time scales, and it is natural to ask whether or not they can be used in cascade.
The aim of this article is to show the validity of such an approximation under a certain condition on the time scales involved in the dynamics, using an averaging procedure.
Then the well-known results of adiabatic theory (see \cite{Bos, BoscainChittaro,Teu}) can be applied in order to get transitions between the eigenstates of the free Hamiltonian.
It leads us to design a control law achieving the inversion of a Spin-$\frac{1}{2}$ particule that is robust with respect to inhomogeneities of the amplitude of the control input (see \cite{Vitanov}). 
Then we can deduce an ensemble controllability result in the sense developed 
in \cite{Li,Beauchard}.
 As a byproduct of the use of a control oscillating with a small frequency detuning, the proposed method is not expected to be robust with respect to inhomogeneities of the resonance frequencies.

\section{NOTATIONS}

Denote by $U(n)$ the Lie group of unitary $n\times n$ matrices and by $u(n)$ its Lie algebra.
For $z\in \C$, denote by $\bar{z}$ its complex conjugate.
For a complex valued matrix $A$, denote by $A_{jk}$ its $(j,k)\in \{1,\dots,n\}^2$ coefficient and 
by $A^{*}$ its adjoint matrix.

\section{GENERAL FRAMEWORK AND MAIN RESULTS}
\subsection{Problem formulation} 
\subsubsection{Rotating frame}
Consider $v,\phi\in C^{\infty}([0,1],\R)$ such that $\phi(0)=0$, $E>0,$ and $\psi_0\in \C^2$.
Denote by $\psi:[0,1]\to \C^2$
 the solution of the equation
\begin{equation}\label{pre-RWA}
i\frac{d\psi(t)}{dt}=\begin{pmatrix}
E&w(t) \\
\bar{w}(t) &-E\\
\end{pmatrix} \psi(t), \quad \psi(0)=\psi_0,
\end{equation} where $w(t)=v( t)e^{i(2Et+\phi(t))}$.
Define $\eta(t)={\bf U}(t)\psi(t)$ where 
\[
{\bf U}(t)=\begin{pmatrix}
e^{-i\left(Et+\frac{\phi(t)}{2}\right)}&0\\
0&e^{i\left(Et+\frac{\phi(t)}{2}\right)}\\
\end{pmatrix}
\]
Then $\eta(t)$ satisfies
\begin{equation}\label{decoupled}
i\frac{d\eta(t)}{dt}=\begin{pmatrix}
-\frac{\phi'(t)}{2}&v(t) \\
v(t) &\frac{\phi'(t)}{2}\\
\end{pmatrix} \eta(t), \quad \eta(0)=\psi_0.
\end{equation}
We say that the dynamics are expressed in the \emph{rotating frame} of speed $E+ \frac{\phi'(t)}{2}$.
Such an equation can be controlled using several approaches, namely
via the well-known Rabi oscillations and the adiabatic approach presented below (see \cite{Vitanov} for a comparison between the two approaches).

\subsubsection{Adiabatic control in the rotating frame}
In order to  design an 
adiabatic control strategy for Equation (\ref{decoupled}),
let us add a parameter $\epsilon$ in the control $w$ and introduce 
$w_{\epsilon}(t)=v(\epsilon t)e^{i(2Et+\frac{\phi(\epsilon t)}{\epsilon})}$.
Consider the corresponding solution of (\ref{decoupled}) with initial condition $\psi_0$, that is,
\begin{equation*}
i\frac{d\eta_{\epsilon}(t)}{dt}= \begin{pmatrix}
-\frac{\phi'(\epsilon t)}{2}&v(\epsilon t) \\
v(\epsilon t) &\frac{\phi'(\epsilon t)}{2}\\
\end{pmatrix} \eta_{\epsilon}(t),\quad \eta_{\epsilon}(0)=\psi_0.
\end{equation*}
In the variable $\tau=\epsilon t\in [0,1]$, the 
reparameterized trajectory $\tilde\eta_{\epsilon}(\tau)=\eta_{\epsilon}(\tau/\epsilon)$ 
satisfies
\begin{equation}\label{AD}
i\frac{d\tilde\eta_{\epsilon}(\tau)}{d\tau}=\frac{1}{\epsilon} \begin{pmatrix}
-\frac{\phi'(\tau)}{2}&v(\tau) \\
v(\tau) &\frac{\phi'(\tau)}{2}\\
\end{pmatrix} \tilde\eta_{\epsilon}(\tau), \quad \tilde\eta_{\epsilon}(0)=\psi_0.
\end{equation}

Let $v$ and $\phi$ be chosen so that the curve  
$(v,\phi'):[0,1]\to \R^2$ connects $(0,-1)$ to $(0,1)$ 
intersecting the vertical axis 
only at its endpoints.  Then, by standard adiabatic approximation, 
if $\psi_0=(1,0)$, then $\tilde\eta_{\epsilon}(1)$ converges, up to phases, to $(0,1)$ as $\epsilon\to0$. 
In the literature, this control strategy, called \emph{chirped adiabatic pulse}, is now very classical. Its robustness properties have been mathematically studied 
in~\cite{Ensemble}.

\subsubsection{Rotating wave approximation}

In many applications only one real control is available. A classical strategy to  duplicate 
the control input is the so-called \emph{rotating wave approximation} (RWA) that works as follows. Let $\varphi_{\tilde\epsilon}:[0,1/\tilde\epsilon]\to\C^2$ be the solution of \eqref{pre-RWA} where $w$ is replaced by the control 
$u_{\tilde\epsilon}(t)=2\tilde\epsilon v(\tilde\epsilon t)\cos(2Et+\phi(\tilde\epsilon t))$.
Let 
\[{\bf U}_{\tilde\epsilon}(t)=\begin{pmatrix}
e^{-i\left(Et+\frac{\phi(\tilde\epsilon t)}{2}\right)}&0\\
0&e^{+i\left(Et+\frac{\phi(\tilde\epsilon t)}{2}\right)}\\
\end{pmatrix}.\]
The RWA then states that $\tau\mapsto {\bf U}_{\tilde\epsilon}(\tau/\tilde\epsilon)\varphi_{\tilde\epsilon}(\tau/\tilde\epsilon)$ 
converges uniformly, as $\tilde\epsilon\to 0$, to the solution of 
\begin{equation}\label{decoupled2}
i\frac{d\eta(t)}{dt}=\begin{pmatrix}
-\frac{\phi'(t)}{2}&v(t) \\
v(t) &\frac{\phi'(t)}{2}\\
\end{pmatrix} \eta(t), \quad \eta(0)=\psi_0.
\end{equation}
Notice that the limit equation \eqref{decoupled2} coincides with \eqref{decoupled}, which is the original equation \eqref{pre-RWA} with complex controls in the rotating frame. 
We have already described how to control \eqref{decoupled} via adiabatic theory. 
It is not clear, however, if the RWA and the adiabatic approximations can be combined. 

For this purpose, we introduce 
$u_{\epsilon_{1},\epsilon_{2}}(t)=2\epsilon_{1}v(\epsilon_{1}\epsilon_{2}t)\cos(2Et+\frac{1}{\epsilon_{2}}\phi(\epsilon_{1}\epsilon_{2}t))$, where $\epsilon_1$ and $\epsilon_2$ play the role of $\tilde\epsilon$ and $\epsilon$, respectively.
In order to establish in which regime the two approximations can be combined, 
 we set $\epsilon_{1}=\epsilon^{\alpha}$, $\epsilon_{2}=\epsilon$ where $\alpha\in \R$ and $u_{\epsilon}=u_{\epsilon^{\alpha},\epsilon}$.
Consider the Cauchy problem 
\begin{equation}\label{twolevel}
i\frac{d\psi_{\epsilon}(t)}{dt}=\begin{pmatrix}
E&u_{\epsilon}(t) \\
u_{\epsilon}(t) &-E\\
\end{pmatrix} \psi_{\epsilon}(t), \quad \psi_{\epsilon}(0)=\psi_0.
\end{equation} 
Define $\Psi_{\epsilon}(t)=U_{\epsilon}(t)\psi_{\epsilon}(t)$ where $U_{\epsilon}(t)=\begin{pmatrix}
e^{-i\left(Et+\frac{\phi(\epsilon^{\alpha+1} t)}{2\epsilon}\right)}&0\\
0&e^{+i\left(Et+\frac{\phi(\epsilon^{\alpha+1}t)}{2\epsilon}\right)}\\
\end{pmatrix}$.
In the variable $\tau=\epsilon^{\alpha+1}t\in [0,1]$, the 
reparameterized trajectory $\tilde\Psi_{\epsilon}(\tau)=\Psi_{\epsilon}(\tau/\epsilon)$ 
satisfies,
\begin{equation}\label{pert}
\frac{d\tilde\Psi_{\epsilon}(\tau)}{d\tau}=\left(\frac{1}{\epsilon}A(\tau)+B_{\epsilon}(\tau)\right) \tilde\Psi_{\epsilon}(\tau), \quad \tilde\Psi_{\epsilon}(0)=\psi_0,
\end{equation}
 where $A(\tau)=-i\begin{pmatrix}
-\phi'(\tau)/2&v(\tau)\\
v(\tau)&\phi'(\tau)/2\\
\end{pmatrix} $ and $B_{\epsilon}(\tau)=
 \frac{-i}{\epsilon} 
\begin{pmatrix}
0&v(\tau)e^{i(\frac{4E\tau}{\epsilon^{\alpha+1}}+\frac{2 \phi(\tau)}{\epsilon})}\\
v(\tau)e^{-i(\frac{4E\tau}{\epsilon^{\alpha+1}} +\frac{2 \phi(\tau)}{\epsilon})}&0\\
\end{pmatrix}$.
The dynamics of $\tilde\Psi$ are characterized by the sum of the term that we had in Equation (\ref{AD}), that corresponds to the dynamics for the complex control case in the rotating frame, and of an oscillating term $B_{\epsilon}(\tau)$. 
The RWA consists in 
neglecting the term $B_{\epsilon}$. We are going to show that this can be mathematically justified if $\alpha>1$. 
 Numerical simulations suggest that 
 the situation is different when the condition is not satisfied.

\subsection{Main results}

In order to obtain the asymptotic analysis announced in the previous section, we show a result of approximation of adiabatic trajectories for general $n$-level systems under the form of Equation (\ref{pert}).
Then we deduce results in the particular case of two-level systems with a drift term.

\subsubsection{Adiabatic approximation result}

\begin{defi}
For $A\in C^\infty([0,1],u(n))$, 
denote by
$j\mapsto \lambda_j(\tau)$ 
the nondecreasing sequence of eigenvalues of $iA(\tau)$.
We say that $A$ satisfies a gap condition if and only if there exists $C>0$ such that 
\begin{equation*}\label{gap}
\forall j\ne k, \forall \tau\in [0,1],\ |\lambda_j(\tau) -\lambda_k(\tau) |>C.
\tag{GAP}
\end{equation*}
\end{defi}

\begin{defi}\label{B}
Let $\alpha$ be a nonzero real number. Define by $S(\alpha)$  the set of families $(B_{\epsilon})_{\epsilon>0}$ of functions in $C^{\infty}([0,1],u(n))$ such that
\begin{itemize}
\item 
$(B_{\epsilon}(\tau))_{jj}=0$ for every $j=1,\dots,n$ and every $\tau\in [0,1]$,
\item 
for every $k>j$ there exist $\beta_{jk}\in\R\setminus\{0\}$ and 
$v_{jk},h_{jk}\in C^{\infty}([0,1],\R)$ such that $(B_{\epsilon}(\tau))_{jk}=-\frac{i}{\epsilon}v_{jk}(\tau)e^{i(\frac{\beta_{jk}\tau}{\epsilon^{\alpha+1}}+\frac{h_{jk}(\tau)}{\epsilon})}$ for every $\tau\in [0,1]$.
\end{itemize} 
\end{defi}

\begin{teo}\label{Drive}
Consider $A\in C^\infty([0,1],u(n))$ and $(B_{\epsilon})_{\epsilon>0}\in S(\alpha)$ with $\alpha>1$. Assume that $A(\cdot)$ satisfies (\ref{gap}). Set $X_0\in \C^n$ independent of $\epsilon$.
Let $X_{\epsilon}$ be the solution of $\frac{dX_{\epsilon}(\tau)}{d\tau}= \left(\frac{1}{\epsilon}A(\tau)+B_{\epsilon}(\tau)\right)X_{\epsilon}(\tau)$ such that $X_{\epsilon}(0)=X_{0}$ and $\hat{X}_{\epsilon}$ be the solution of $\frac{d\hat{X}_{\epsilon}(\tau)}{d\tau}= \frac{1}{\epsilon}A(\tau)\hat{X}_{\epsilon}(\tau)$ such that $\hat{X}_{\epsilon}(0)=X_{0}$.
Then there exists $c>0$ independent of $\tau,\epsilon$ such that for every $\tau \in [0,1]$, $\|X_{\epsilon}(\tau)-\hat{X}_{\epsilon}(\tau)\|\leq c \epsilon^{\text{min}(1,\alpha-1)}$.
\end{teo}

\subsubsection{Application to two-level systems}
We consider $v,\phi\in C^{\infty}([0,1],\R)$ such that $\phi(0)=0$ and $E>0$.
We consider now Equation~(\ref{twolevel}) where $u_{\epsilon}(t)=2\epsilon^{\alpha}v(\epsilon^{\alpha+1}t)\cos(2Et+\frac{1}{\epsilon}\phi(\epsilon^{\alpha+1}t))$.
In the fast time scale $\tau=\epsilon^{\alpha+1}t\in [0,1]$, Equation~(\ref{twolevel}) can be rewritten as 
\begin{equation}\label{spin}
i\frac{d\psi_{\epsilon}(\tau)}{d\tau}=\begin{pmatrix}
\frac{E}{\epsilon^{\alpha+1}}&u_{\epsilon}(\tau) \\
u_{\epsilon}(\tau) &-\frac{E}{\epsilon^{\alpha+1}}\\
\end{pmatrix} \psi_{\epsilon}(\tau)
\end{equation} for $\tau\in [0,1]$ where by a slight abuse of notation, we write $u_{\epsilon}(\tau)=\frac{2}{\epsilon}v(\tau)\cos(\frac{2E\tau}{\epsilon^{\alpha+1}}+\frac{1}{\epsilon}\phi(\tau))$.
Set $\psi_0\in \C^2$ independent of $\epsilon$.
Let $\psi_{\epsilon}(\cdot)$ be the solution of Equation~(\ref{spin}) such that  $\psi_{\epsilon}(0)=\psi_0$.
Similarly, let $\hat{\psi}_{\epsilon}$ be the solution of
\begin{equation}\label{spincomplexe}
i\frac{d\hat{\psi}_{\epsilon}(\tau)}{d\tau}=\begin{pmatrix}
\frac{E}{\epsilon^{\alpha+1}}&w_{\epsilon}(\tau) \\
\bar{w}_{\epsilon}(\tau) &-\frac{E}{\epsilon^{\alpha+1}}\\
\end{pmatrix} \hat{\psi}_{\epsilon}(\tau), \quad \hat{\psi}_{\epsilon}(0)=\psi_0
\end{equation} for $\tau\in [0,1]$ and $w_{\epsilon}(\tau)=\frac{1}{\epsilon}v(\tau)e^{i(\frac{2E\tau}{\epsilon^{\alpha+1}}+\frac{1}{\epsilon}\phi(\tau))}$.

\begin{teo}\label{RWA}
Assume that $\alpha>1$. 
Consider $v,\phi$ in $C^{\infty}([0,1],\R)$ such that $\phi(0)=0$ and $v^2+\frac{\phi'^2}{4}$ is bounded from below by $C>0$.
Then the solution  $\psi_{\epsilon}$ of Equation (\ref{spin}) satisfies
$\|\psi_{\epsilon}(\tau) -\hat{\psi}_{\epsilon}(\tau)\|<c \epsilon^{\text{min}(1,\alpha-1)}$ where
 $c>0$ is independent of $(\tau,\epsilon)$.
\end{teo}

Theorem~\ref{RWA} will be used  in Section \ref{Spinup} to design control laws for two-level systems using the key fact that $\hat{\psi}_{\epsilon}(\tau)$ follows an adiabatic evolution up to a change of frame.

\section{APPROXIMATION RESULTS}

\subsection{Variation formula}

We recall here without proof 
a classical formula which will be useful to neglect highly oscillating parts of the dynamics.
\begin{prop}[Variation formula \cite{Agra}]\label{variation}
Consider 
\begin{equation}\label{vari}
\frac{dx(\tau)}{d\tau}=\left(A(\tau)+B(\tau)\right)x(\tau), \quad x(\tau)\in \C^n,
\end{equation}
where $A,B$ be in $C^\infty([0,1],u(n))$. 
Denote the flow at time $\tau$  of 
$\frac{dx(\tau)}{d\tau}=A(\tau)x(\tau)$  by $P_{\tau}\in U(n)$ and the flow at time $\tau$ of 
$\frac{dx(\tau)}{d\tau}=P_{\tau}^{-1}B(\tau) P_{\tau}x(\tau)$  by $W_{\tau}\in U(n)$.
Then the flow of 
(\ref{vari}) at time $\tau$ is equal to $
P_{\tau}W_{\tau}$.
\end{prop}

\subsection{Regularity of the eigenstates}
We recall here a 
well-known regularity result. 
\begin{lem}\label{reg}
Let $A\in C^\infty([0,1],u(n))$
satisfy (\ref{gap}).
Then the eigenvectors and the eigenvalues of $iA(\tau)$ can be chosen $C^{\infty}$ with respect to $\tau$.
\end{lem}

\subsection{Averaging of quantum systems}

\begin{teo}\label{av}
Consider $A$ and $(A_{\epsilon})_{\epsilon>0}$ in $C^\infty([0,1],u(n))$ and assume that 
$A_{\epsilon}(\tau)$ is uniformly bounded w.r.t. $(\tau,\epsilon)$.
Denote the flow of the equation $\frac{dx(\tau)}{d\tau}=A(\tau)x(\tau)$ at time $\tau$  by $P_{\tau}\in U(n)$ and the flow of the equation $\frac{dx(\tau)}{d\tau}=A_{\epsilon}(\tau)x(\tau)$ at time $\tau$  by $P_{\tau}^{\epsilon}\in U(n)$.
If $\int_0^\tau A_{\epsilon}(s)ds=\int_0^\tau A(s) ds +O(\epsilon)$,
then $P_{\tau}^{\epsilon}=P_{\tau}+O(\epsilon)$,
both estimates being uniform w.r.t. $\tau\in  [0,1]$.
\end{teo}

We state Theorem~\ref{av} without proof because it is a particular case of next result, Theorem~\ref{avnb}.
In the following, we do not assume the boundedness of $A_{\epsilon}$ with respect to $\epsilon$.
We refer to \cite{Kurzweil87,Kurzweil88bis,Kurzweil88,Liu97,Sussmann1993} for more informations on the case of averaging of a general class of dynamical systems with non-bounded and highly oscillatory inputs.
Our result provides an estimate of the error in the special case of quantum systems.
\begin{teo}\label{avnb}
Consider $A$ and $(B_{\epsilon})_{\epsilon>0}$ in $C^\infty([0,1],u(n))$.
Assume that
$\int_0^\tau B_{\epsilon}(s)ds=O(\epsilon)$
and 
$\int_0^\tau \left|B_{\epsilon}(s)\right|\left|\int_0^s B_{\epsilon}(x)dx\right| 
ds=O(\epsilon^k)$
uniformly w.r.t. $\tau\in [0,1]$, with $k>0$. 
Set $A_{\epsilon}=A+B_{\epsilon}$.
Denote the flow of the equation $\frac{dx(\tau)}{d\tau}=A(\tau)x(\tau)$ at time $\tau$  by $P_{\tau}\in U(n)$ and the flow of the equation $\frac{dx(\tau)}{d\tau}=A_{\epsilon}(\tau)x(\tau)$ at time $\tau$  by $P_{\tau}^{\epsilon}\in U(n)$.
Then we have 
$P_{\tau}^{\epsilon}=P_{\tau}+O(\epsilon^{\text{min}(k,1)})$,
uniformly w.r.t. $\tau\in  [0,1]$.

\begin{proof}
Under the hypotheses of the theorem, there exists $K>0$ such that for every $\tau\in [0,1]$,
 $|\int_0^\tau B_\epsilon(s) ds|<K \epsilon$. 
Let $Q_{\tau}^{\epsilon}$ be the flow associated with $B_{\epsilon}$.
We have $Q_{\tau}^{\epsilon}=\text{Id}+\int_0^\tau B_{\epsilon}(s)Q_{s}^{\epsilon} ds$, where $\text{Id}$ is the identity $n\times n$ matrix.
By integration by parts, $Q_{\tau}^{\epsilon}=\text{Id}+\left(\int_0^\tau B_{\epsilon}(s) ds\right) Q_{\tau}^{\epsilon} - \int_0^\tau (\int_0^s B_{\epsilon}(\theta) d\theta)B_{\epsilon}(s)Q_{s}^{\epsilon} ds$.
Moreover,
$Q_{\tau}^{\epsilon}$ is bounded uniformly w.r.t. $(\tau,\epsilon)$, since it evolves in $U(n)$.
By the triangular inequality, we get 
\begin{align*} \left|Q_{\tau}^{\epsilon}-\text{Id}\right|\leq &  \left|\int_0^\tau B_{\epsilon}(s) ds\right| \left|Q_{\tau}^{\epsilon}\right| + \int_0^\tau  \left|\int_0^s B_{\epsilon}(\theta) d\theta\right|\left| B_{\epsilon}(s)Q_{s}^{\epsilon}\right| ds\\
\leq& C_1 \epsilon + C_2 \epsilon^k,
\end{align*}
where $C_1,C_2$ are positive constants which do not depend on $(\epsilon,\tau)$. Hence, we deduce that
$Q_{\tau}^{\epsilon}=\text{Id}+O(\epsilon^q)$ uniformly w.r.t. $\tau \in [0,1]$, where $q=\text{min}(k,1)$.
The variation formula (Proposition~\ref{variation}) provides $P_{\tau}^{\epsilon}=Q_{\tau}^{\epsilon}W_{\tau}^{\epsilon}$, where $W_{\tau}^{\epsilon}\in U(n)$ is the flow of the equation $\frac{dx(\tau)}{d\tau}=(Q_{\tau}^{\epsilon})^{-1} A(\tau) Q_{\tau}^{\epsilon}x(\tau)$ at time $\tau$. By the previous estimate, we have $(Q_{\tau}^{\epsilon})^{-1} A(\tau) Q_{\tau}^{\epsilon}=A(\tau)+O(\epsilon^{q})$ uniformly w.r.t. $\tau \in [0,1]$.
By 
Gronwall's Lemma, we get that $W_{\tau}^{\epsilon}=P_{\tau}+O(\epsilon^{q})$  and we can conclude.
\end{proof}
\end{teo}

\subsection{Perturbation of an adiabatic trajectory}\label{DRIVING}
Consider $A,B_{\epsilon}\in C^\infty([0,1],u(n))$.
Fix $\psi_0\in \C^n$. Let $X_{\epsilon}$ be the solution of $\frac{dX_{\epsilon}(\tau)}{d\tau}= \left(\frac{1}{\epsilon}A(\tau)+B_{\epsilon}(\tau)\right)X_{\epsilon}(\tau)$ such that $X_{\epsilon}(0)=X_{0}$ and let $\hat{X}_{\epsilon}$ be the solution of $\frac{d\hat{X}_{\epsilon}(\tau)}{d\tau}= \frac{1}{\epsilon}A(\tau)\hat{X}_{\epsilon}(\tau)$ such that $\hat{X}_{\epsilon}(0)=X_{0}$, that we call the \emph{adiabatic trajectory associated with $A$}.
The goal of this section is to understand under which conditions on $B_{\epsilon}(\cdot)$ we have 
\begin{equation}
\tag{T}
 \|X_{\epsilon}(\tau)-\hat{X}_{\epsilon}(\tau)\|\to 0
\label{insensitive}
\end{equation} 
uniformly with respect to $\tau\in [0,1]$.
By the variation formula (Proposition~\ref{variation}), one can show that if the flow of $\frac{dx(\tau)}{d\tau}=B_{\epsilon}(\tau)x(\tau)$, $x(\tau)\in \C^n$, is equal to $\text{Id}+O(\epsilon^k)$ uniformly w.r.t. $\tau\in [0,1]$ with $k>1$, then Property (\ref{insensitive}) is satisfied. However this condition is too conservative for our needs. 
We restrict our study to the class of perturbations  $(B_{\epsilon})_{\epsilon>0}\in S(\alpha)$ introduced in the Definition \ref{B}.
 We give below a  a sufficient condition on $\alpha$ such that Property (\ref{insensitive}) is satisfied for every $A$ satisfying Condition (\ref{gap}) and every $(B_{\epsilon})_{\epsilon>0}\in S(\alpha)$ (Proposition~\ref{RWAveraging}).
Based on such a result we then provide a proof of Theorem~\ref{Drive}.

\begin{lem} \label{fast}
For every $\alpha>0$ and every  $a,h\in C^{\infty}([0,1],\R)$, we have  $\int_0^\tau a(s) e^{i(\frac{s}{\epsilon^{\alpha+1}}+\frac{h(s)}{\epsilon})} ds=O(\epsilon^{\alpha+1})$ uniformly with respect to $\tau\in [0,1]$.
\end{lem}
\begin{proof}
 Integrating by parts, for every $\tau \in [0,1],$
\begin{align*}
&\int_0^\tau a(s) e^{i(\frac{s}{\epsilon^{\alpha+1}}+\frac{h(s)}{\epsilon})} ds \\
&=i \epsilon^{\alpha+1}\int_0^\tau e^{i \frac{s}{\epsilon^{\alpha+1}}} \left( a'(s)+i \frac{h'(s)}{\epsilon}a(s)\right)e^{i\frac{h(s)}{\epsilon}}ds\\
&+ \left[-i \epsilon^{\alpha+1} e^{i\frac{s}{\epsilon^{\alpha+1}}}a(s)e^{i\frac{h(s)}{\epsilon}}\right]_{0}^{\tau}\\
&=-\epsilon^{\alpha} \int_0^\tau h'(s)a(s) e^{i \frac{s}{\epsilon^{\alpha+1}}}   e^{i\frac{h(s)}{\epsilon}}ds +O(\epsilon^{\alpha+1}).
\end{align*}
Iterating the integration by parts on the integral term $\lceil\frac{1}{\alpha}\rceil$ more times, we get
$\int_0^\tau a(s) e^{i(\frac{s}{\epsilon^{\alpha+1}}+\frac{h(s)}{\epsilon})} ds=O(\epsilon^{\alpha+1})$.
\end{proof}

\begin{defi}\label{def:MP}
Let $\alpha>0$ and $(B_\epsilon)_{\epsilon>0}$ be in $S(\alpha)$.
For every $\epsilon>0$, $P\in C^\infty([0,1],U(n))$, and every diagonal matrix $\Gamma(\tau)=\text{diag}(\Gamma_j(\tau))_{j=1}^n$ with 
$\Gamma_j\in C^\infty([0,1],\R)$, $j=1,\dots,n$,
define 
\[M(P,\Gamma,\epsilon)(\tau)=e^{i\frac{\Gamma(\tau)}{\epsilon}}P^{*}(\tau)B_{\epsilon}(\tau)P(\tau)e^{-i\frac{\Gamma(\tau)}{\epsilon}}, \quad\tau\in [0,1].\]
\end{defi}

\begin{lem}\label{conjug}

Let $\alpha>1$. Consider $(B_\epsilon)_{\epsilon>0}$, 
$P$, $\Gamma$, and $M$ as in Definition~\ref{def:MP}.
Then $\int_0^\tau M(P,\Gamma,\epsilon)(s) ds=O(\epsilon^{\alpha})$ and 
$\int_0^\tau \left|M(P,\Gamma,\epsilon)(s)\right|\left|\int_0^s M(P,\Gamma,\epsilon)(x)dx\right| 
ds=O(\epsilon^{\alpha-1})$ uniformly w.r.t. $\tau\in [0,1]$.
\end{lem}

\begin{proof}
Define the following matrix $C_{\epsilon}(\tau)=\frac{1}{\epsilon}v_{j\ell}(\tau)e^{i
(\frac{\beta_{j\ell}\tau}{\epsilon^{\alpha+1}}+\frac{h_{j\ell}(\tau)}{\epsilon}
)}E_{j\ell}$  for fixed $j,\ell\in \{1,\dots,n\}$ where $E_{j\ell}$ is the matrix whose coefficient $(j,\ell)$ is equal to $1$ and others are equal to $0$.
By direct computations, denoting $(P(\tau))_{kq}=p_{kq}(\tau),$ we get 
\begin{align*}
&e^{i\frac{\Gamma(s)}{\epsilon}}P^{*}(s)C_{\epsilon}(s)P(s)e^{-i\frac{\Gamma(s)}{\epsilon}}\\
&=\frac{v_{j\ell}(\tau)}{\epsilon}\sum_{k,q=1}^n p_{\ell k}(\tau)\bar{p}_{jq}(\tau)e^{\frac{i}{\epsilon} (\Gamma_q(\tau)-\Gamma_k(\tau))}e^{i(\frac{\beta_{j\ell}\tau}{\epsilon^{\alpha+1}}+\frac{h_{j\ell}(\tau)}{\epsilon})}E_{qk}.
\end{align*}
By Lemma \ref{fast}, we get for every $q,k \in \{1,\dots,n\}$, \[\int_0^\tau v_{j\ell}(s)  p_{\ell k}(s)\bar{p}_{jq}(s)e^{\frac{i}{\epsilon} (\Gamma_q(s)-\Gamma_k(s))}e^{i(\frac{\beta_{j\ell}s}{\epsilon^{\alpha+1}}+\frac{h_{j\ell}(s)}{\epsilon})}ds\] is $O(\epsilon^{\alpha+1})$.
Hence, 
$\int_0^\tau  e^{i\frac{\Gamma(s)}{\epsilon}}P(s)C_{\epsilon}(s)P^{*}(s)e^{-i\frac{\Gamma(s)}{\epsilon}} ds=O(\epsilon^{\alpha})$. 
We deduce by linearity that the result is also true for $B_{\epsilon}$.
The last 
claim follows noticing that $M(P,\Gamma,\epsilon)(\tau)=O(\frac{1}{\epsilon})$.
\end{proof}

\begin{lem}\label{killoscillations}
Let $\alpha>1$. Consider $(B_\epsilon)_{\epsilon>0}$, 
$P$, $\Gamma$, and $M$ as in Definition~\ref{def:MP}.
Then the flow of 
$\frac{dx(\tau)}{d\tau}=M(P,\Gamma,\epsilon)(\tau)x(\tau)$ 
is equal to 
$\text{Id}+O(\epsilon^{\alpha-1})$, 
uniformly w.r.t. $\tau \in [0,1]$.
\end{lem}
\begin{proof}
We apply Theorem \ref{avnb} using the estimates from Lemma~\ref{conjug}.
\end{proof}

The next proposition, based on Lemma \ref{killoscillations}, shows that under the condition $\alpha>1$, an adiabatic trajectory is robust with respect to perturbations of the dynamics by a term of the form $(B_{\epsilon})_{\epsilon>0}\in S(\alpha)$ for $\epsilon$ small. 
\begin{prop}\label{RWAveraging}
Consider  $A\in C^\infty([0,1],u(n))$ and $(B_{\epsilon})_{\epsilon>0}\in S(\alpha)$
with $\alpha>1$.
Assume that Condition (\ref{gap}) is satisfied.
Select $\lambda_j\in C^\infty([0,1],\R)$, $j=1,\dots,n$,  and $P\in C^\infty([0,1],U(n))$  such that, for $j=1,\dots,n$, $\lambda_j(\tau)$  and the $j$-th column of $P(\tau)$ are, respectively,  an eigenvalue of $iA(\tau)$ and a corresponding eigenvector (the  existence of $C^\infty$ eigenpairs being guaranteed by 
Lemma~\ref{reg}).
Define
 $\Lambda(\tau)=\text{diag}(\lambda_j(\tau))_{j=1}^n$, $\tau\in [0,1]$.
Fix $X_0\in \C^n$ independent of $\epsilon$.
Let $X_{\epsilon}$ be the solution of $\frac{dX_{\epsilon}(\tau)}{d\tau}= \left(\frac{1}{\epsilon}A(\tau)+B_{\epsilon}(\tau)\right)X_{\epsilon}(\tau)$ such that $X_{\epsilon}(0)=X_{0}$.
Set $\Upsilon_{\epsilon}(\tau)=P(\tau) \exp\left(\frac{-i}{\epsilon} \int_0^{\tau} \Lambda(s) ds\right)  \exp\left(\int_{0}^{\tau} D(s)ds\right) P^{*}(0)$ where $D$ is equal to the diagonal part of $\frac{dP^{*}}{d\tau}P$. 
Then
 $\|X_{\epsilon}(\tau) -\Upsilon_{\epsilon}(\tau) X_0\|<c \epsilon^{\text{min}(1,\alpha-1)}$
for some constant $c>0$ independent of $\tau\in [0,1]$ and $\epsilon>0$.

\begin{proof}
Define $\Gamma(\tau)=\int_0^\tau \Lambda(s) ds$ and $Y_{\epsilon}(\tau)=\exp\left(\frac{i}{\epsilon} \Gamma(\tau) \right) P^{*}(\tau) X_{\epsilon} (\tau)$. 
Then $Y_{\epsilon}$ satisfies the equation
\begin{equation}\label{eqpert}
 \begin{split}
 &\frac{dY_{\epsilon}(\tau)}{d\tau}=M(P,\Gamma,\epsilon)(\tau)Y_{\epsilon}(\tau)\\
 +&\exp\left(\frac{i}{\epsilon} \Gamma(\tau)\right) \frac{dP^{*}}{d\tau}(\tau) P(\tau) \exp\left(-\frac{i}{\epsilon}\Gamma(\tau) \right)Y_{\epsilon}(\tau),
 \end{split}
\end{equation} 
where $M(P,\Gamma,\epsilon)$ is defined as in Definition~\ref{def:MP}.
In order to simplify the notations, set $D_{\epsilon}(\tau)=\exp\left(\frac{i}{\epsilon} \Gamma(\tau)\right) \frac{dP^{*}}{d\tau}(\tau) P(\tau) \exp\left(-\frac{i}{\epsilon} \Gamma(\tau) \right)$ and denote the flow at time $\tau$ of the equations $\frac{dx(\tau)}{d\tau}= M(P,\Gamma,\epsilon)(\tau)x(\tau)$  and $\frac{dx(\tau)}{d\tau}=(P_{\tau}^{\epsilon})^{-1} D_{\epsilon}(\tau) P_{\tau}^{\epsilon}x(\tau)$ by $P_{\tau}^{\epsilon}$ and  $W_{\tau}^{\epsilon}$, respectively.
By the variation formula (Proposition~\ref{variation}), we get that the flow at time $\tau$ of  equation (\ref{eqpert}) is equal to $Q_{\tau}^{\epsilon} =P_{\tau}^{\epsilon}W_{\tau}^{\epsilon} $.
By Lemma \ref{killoscillations}, we have $P_{\tau}^{\epsilon}=\text{Id} +O(\epsilon^{\alpha-1})$.
Hence $(P_{\tau}^{\epsilon})^{-1} D_{\epsilon}(\tau)P_{\tau}^{\epsilon}=D_{\epsilon}(\tau)+O(\epsilon^{\alpha-1})$.
Using the gap condition (\ref{gap}), we have the estimate $\int_0^\tau D_{\epsilon}(s)ds=\int_0^\tau D(s) ds+O(\epsilon)$ uniformly with respect to $\tau \in [0,1]$.
Indeed, 
$(D_{\epsilon}(\tau))_{jl}=q_{jl}(\tau)e^{\frac{i}{\epsilon}\int_0^\tau (\lambda_j(s)-\lambda_l(s))ds}$, $j,l\in\{1,\dots,n\}$, where $q_{jl}$ is $C^{\infty}$. Hence we get the expected estimation by a direct estimation of the integral of the oscillating term $e^{\frac{i}{\epsilon}\int_0^\tau (\lambda_j(s)-\lambda_l(s))}ds$, $j,l\in\{1,\dots,n\}$.
Moreover, since $D_{\epsilon}$ is bounded with respect to $\epsilon$, Theorem \ref{av} ensures that  $W_{\tau}^{\epsilon}=\exp \left(\int_0^\tau D(s) ds\right)+O(\epsilon^{\text{min}(1,\alpha-1)})$.
It follows that 
\begin{align*}
Q_{\tau}^{\epsilon} =&\left(\text{Id}+O(\epsilon^{\alpha-1})\right)\left( \exp \left(\int_0^\tau D(s) ds\right)+O(\epsilon^{\text{min}(1,\alpha-1)})\right)\\
=&\exp \left(\int_0^\tau D(s) ds\right)+O(\epsilon^{\text{min}(1,\alpha-1)}).
\end{align*}
\end{proof}
\end{prop}

\begin{proof}(Proof of Theorem \ref{Drive})
By an easy application of Theorem \ref{av}, we get the adiabatic estimate
$\forall \tau\in [0,1], \|\hat{X}_{\epsilon}(\tau) -\Upsilon_{\epsilon}(\tau) X_0\|<C \epsilon$, where $C>0$ is independent of $\tau,\epsilon$ and $\Upsilon_{\epsilon}$ is defined as in Proposition~\ref{RWAveraging}.
The result is then obtained combining the previous inequality with the estimate of Proposition \ref{RWAveraging} by triangular inequality.
\end{proof}

\subsection{1-parameter family case}

\begin{defi}
For $A^{\delta}(\tau)\in u(n)$ whose dependence on $(\tau,\delta)\in [0,1]\times [a,b]$ is $C^{\infty}$, define $\Lambda^{\delta}(\tau)=\text{diag}(\lambda_j^{\delta}(\tau))_{j\in \{1,\dots,n\}}$ where  $j\mapsto \lambda_j^{\delta}(\tau)$ is the nondecreasing sequence of eigenvalues of $iA^{\delta}(\tau)$.
We say that $A$ satisfies a uniform gap condition if 
there exists $C>0$ such that 
	\begin{equation*}\label{ugap}
	\tag{UGAP}
	\forall k\ne j, \forall \delta\in [a,b], \forall \tau\in [0,1],\ |\lambda_k^{\delta}(\tau) -\lambda_j^{\delta}(\tau) |>C.
	\end{equation*}
\end{defi}
Using uniform estimates with respect to $\delta\in [a,b]$ in the proof of Proposition \ref{RWAveraging}, we get the following theorem.
\begin{teo}\label{UDrive}
Consider $(B_{\epsilon})_{\epsilon>0}\in S(\alpha)$ with $\alpha>1$.
Let $(A^{\delta}(\tau))_{\delta\in[a,b]}$ be a family of matrices in $u(n)$ whose dependence in $(\tau,\delta)\in [0,1]\times [a,b]$ is $C^{\infty}$. Assume that $A(\tau)$ satisfies (\ref{ugap}).
Fix $X_0\in \C^n$ independent of $\epsilon$.
Let $X_{\epsilon}(\delta,\tau)$ be the solution of $\frac{dX_{\epsilon}(\tau)}{d\tau}= \left(\frac{1}{\epsilon}A^{\delta}(\tau)+\delta B_{\epsilon}(\tau)\right)X_{\epsilon}(\tau)$ such that $X_{\epsilon}(0)=X_{0}$ and $\hat{X}_{\epsilon}(\delta,\tau)$ be the solution of $\frac{d\hat{X}_{\epsilon}(\tau)}{d\tau}= \frac{1}{\epsilon}A^{\delta}(\tau)\hat{X}_{\epsilon}(\tau)$ such that $\hat{X}_{\epsilon}(0)=X_{0}$.
Then there exists $c>0$ independent of $\tau, \delta, \epsilon$ such that for every $(\tau,\delta) \in [0,1]\times [a,b]$, $\|X_{\epsilon}(\delta,\tau)-\hat{X}_{\epsilon}(\delta,\tau)\|\leq c \epsilon^{\text{min}(1,\alpha-1)}$.
\end{teo}

\section{Control of two-level systems}\label{Spinup}

We start this section proving Theorem \ref{RWA}.

\begin{proof}(Proof of Theorem \ref{RWA}) Apply the unitary transformation
$X_{\epsilon}(\tau)=U_{\epsilon}(\tau)\psi_{\epsilon}(\tau)$ where 
\[U_{\epsilon}(\tau)=\begin{pmatrix}
e^{-i\left(\frac{E\tau}{\epsilon^{\alpha+1}}+\frac{\phi(\tau)}{2\epsilon}\right)}&0\\
0&e^{i\left(\frac{E\tau}{\epsilon^{\alpha+1}}+\frac{\phi(\tau)}{2\epsilon}\right)}\\
\end{pmatrix}.\]
Then $X_{\epsilon}$ satisfies 
$\frac{dX_{\epsilon}(\tau)}{d\tau}=\left(\frac{1}{\epsilon}A(\tau)+B_{\epsilon}(\tau)\right) X_{\epsilon}(\tau)$, $X_{\epsilon}(0)=\psi_0$, 
 where $A(\tau)=-i\begin{pmatrix}
-\phi'(\tau)/2&v(\tau)\\
v(\tau)&\phi'(\tau)/2\\
\end{pmatrix} $ and $B_{\epsilon}(\tau)=
 \frac{-i}{\epsilon} 
\begin{pmatrix}
0&v(\tau)e^{i(\frac{4E\tau}{\epsilon^{\alpha+1}}+\frac{2 \phi(\tau)}{\epsilon})}\\
v(\tau)e^{-i(\frac{4E\tau}{\epsilon^{\alpha+1}} +\frac{2 \phi(\tau)}{\epsilon})}&0\\
\end{pmatrix}$.
The condition  $v(\tau)^2+\frac{\phi'(\tau)^2}{4}>C$ for every $\tau\in [0,1]$ implies that $A$ satisfies 
Condition (\ref{gap}).
Let $\hat{X}_{\epsilon}:[0,1]\to \C^2$ be the solution of 
$\frac{d\hat{X}_{\epsilon}(\tau)}{d\tau}=\frac{1}{\epsilon}A(\tau) \hat{X}_{\epsilon}(\tau),
\quad \hat{X}_{\epsilon}(0)=\psi_0$.
Theorem \ref{Drive} 
then implies that $\|X_{\epsilon}(\tau)-\hat{X}_{\epsilon}(\tau)\|\leq c \epsilon^{\text{min}(1,\alpha-1)} $.
Noticing that $\tilde{\psi}_{\epsilon}(\tau)=U_{\epsilon}^{*}(\tau)\hat{X}_{\epsilon}(\tau)$, we get the result.
\end{proof}

\subsection{Control strategy for two-level systems and simulations}
\label{control}
 Let $v,\phi\in C^{\infty}([0,1],\R)$ be such that $v(0)=v(1)=0$, 
 $\phi(0)=0$, $\phi'(0)\phi'(1)<0$,  and  
 $v(\tau)\neq 0$ for $\tau \in (0,1)$.
Let $e_1=\begin{pmatrix}
1\\
0
\end{pmatrix}$ and $e_2=\begin{pmatrix}
0\\
1
\end{pmatrix}$.
By adiabatic approximation, the solution $x_{\epsilon}:[0,1]\to\C^2$ of 
\begin{equation*}
\frac{dx_{\epsilon}(\tau)}{d\tau}=-\frac{i}{\epsilon}\begin{pmatrix}
-\phi'(\tau)/2&v(\tau)\\
v(\tau)&\phi'(\tau)/2\\
\end{pmatrix} x_{\epsilon}(\tau),
\quad x_{\epsilon}(0)=e_1,
\end{equation*}
satisfies
$\|x_{\epsilon}(1)-e^{i\xi_{\epsilon}}e_2\|\leq C \epsilon$ where $C>0$ is independent of $\epsilon$ and $\xi_{\epsilon}\in \R$.
Consider the solution $\psi_{\epsilon}(\tau)$ of Equation (\ref{spin}) such that
 $\psi_0=e_1$ and corresponding to the controls $(v,\phi)$.
Applying Theorem \ref{RWA}, we have 
$\|\psi_{\epsilon}(1)-e^{i\theta_{\epsilon}}e_2\|\leq c \epsilon^{\text{min}(1,\alpha-1)}$ where $c>0$ is independent of $\epsilon$ and $\theta_{\epsilon}\in \R$.

On Figure \ref{realcontrol}, we have plotted the projection of the wave function onto 
$e_2$
for $v(\tau)=\sin(\pi \tau), \quad \phi(\tau)=-\frac{1}{\pi}\sin(\pi \tau)$,
with $\epsilon=0.01$, $\alpha=1.5$ and $E=1$ 
in the fast time scale, that is,  
as a function of 
$\tau\in [0,1]$.
The total time needed by our control strategy in the 
variable $t=\frac{\tau}{\epsilon^{\alpha+1}}$ is $T=
\frac{1}{\epsilon^{\alpha+1}}$.
On Figure \ref{diff}, we have plotted the norm of the difference between $\psi_{\epsilon}$ and the solution of Equation (\ref{spincomplexe}) with the same initial condition and parameters as a function of $\tau\in [0,1]$.
\begin{figure}
\begin{center}
\includegraphics[scale=0.4]{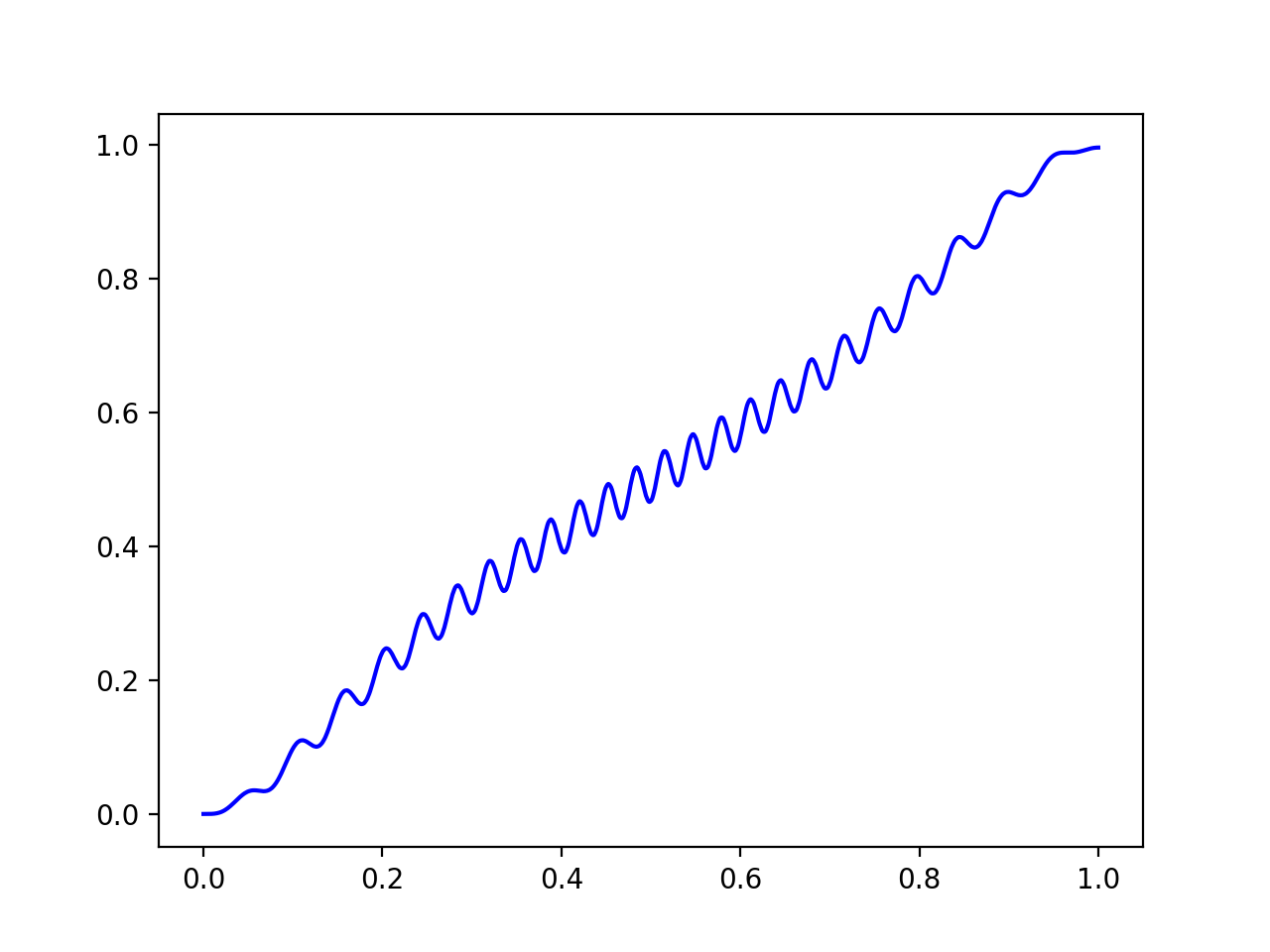}
\caption{Fidelity $|\langle \psi_{\epsilon}(\tau) , e_2
\rangle|^2$ as a function of the time variable $\tau\in [0,1]$ with $\epsilon=0.01$, $\alpha=1.5$, and $E=1$.}
\label{realcontrol}
\end{center}
\end{figure}
\begin{figure}
\begin{center}
\includegraphics[scale=0.35]{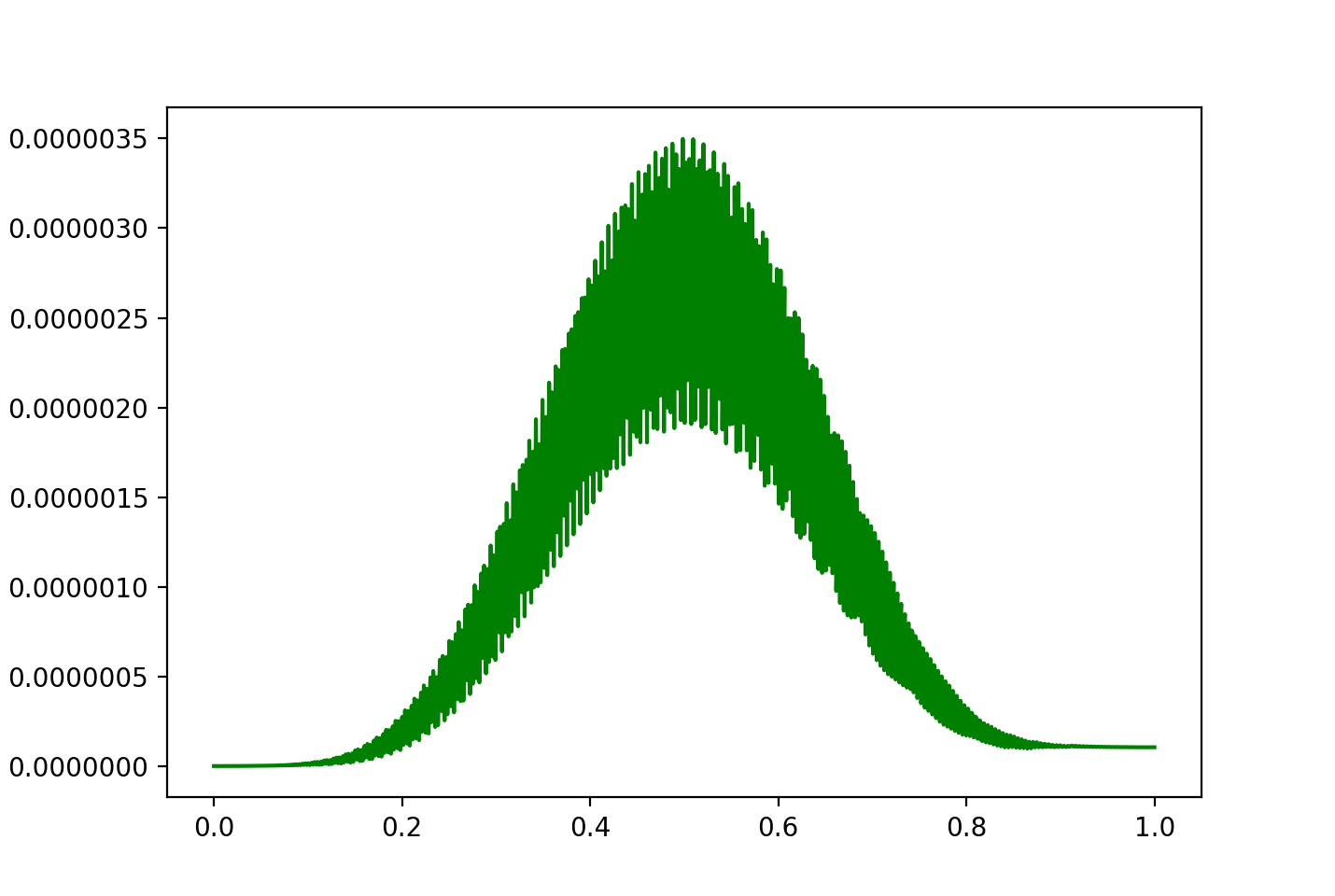}
\caption{Squared norm of the difference between 
$\psi_\epsilon$ and $\hat\psi_\epsilon$  as a function of the time variable $\tau\in [0,1]$
with $\epsilon=0.01$, $\alpha=1.5$, and $E=1$.}
\label{diff}
\end{center}
\end{figure}
\subsection{Robustness of the control strategy with respect to amplitude of control inhomogeneities}
Let $U$ be a connected open set of $\R$ containing $0$.
\begin{teo}
Let $E\in\R\setminus\{0\}$.
The equation 
\begin{equation}\label{robust}
i\frac{d\psi}{dt}=\begin{pmatrix}
E&\delta u\\
\delta u&-E
\end{pmatrix}
\psi
\end{equation}
is approximately ensemble controllable between the eigenstates of $H_0=\begin{pmatrix}
E&0\\
0&-E
\end{pmatrix}$
uniformly with respect to $\delta \in [a,b]\subset (0,+\infty)$ and $u\in U$, that is,
 for every $\epsilon>0$ there exist $T>0$ and $u\in L^{\infty}([0,T],U)$ 
 such that, for every $\delta\in [a,b]$, the solution $\psi(\delta,\cdot):[0,T] \to \C^2$ of Equation~(\ref{robust}) with initial condition  $\psi(\delta,0)=e_1$  satisfies  $\forall \delta\in [a,b], \|\psi(\delta,T)-e^{i\theta_{\delta}}e_2\|<\epsilon$ where $\theta_{\delta}\in \R$.
 
\begin{proof}
Let $\alpha>1$ and $v,\phi\in C^{\infty}([0,1],\R)$ be such that $v(0)=v(1)=0$,  $\phi(0)=0$, $\phi'(0)\phi'(1)<0$, and $v(\tau)\neq 0$ for $\tau \in (0,1)$.
Let us consider $\tau=\epsilon^{\alpha+1}t\in [0,1]$ and $u_{\epsilon}(\tau)=\frac{2}{\epsilon}v(\tau)\cos(\frac{2E\tau}{\epsilon^{\alpha+1}}+\frac{1}{\epsilon}\phi(\tau))$.
For each $\delta\in [a,b]$, let $\psi_{\epsilon}(\delta,\tau)$ be the solution of 
\begin{equation*}
i\frac{d\psi(\delta,\tau)}{d\tau}=\begin{pmatrix}
\frac{E}{\epsilon^{\alpha+1}}&\delta u_{\epsilon}(\tau)\\
\delta u_{\epsilon}(\tau)&-\frac{E}{\epsilon^{\alpha+1}}
\end{pmatrix}
\psi(\delta,\tau),\quad \psi_{\epsilon}(\delta,0)=e_1.
\end{equation*} 
Apply the unitary transformation
$X_{\epsilon}(\delta,\tau)=U_{\epsilon}(\tau)\psi_{\epsilon}(\delta,\tau)$ where $U_{\epsilon}(\tau)=\begin{pmatrix}
e^{-i\left(\frac{E\tau}{\epsilon^{\alpha+1}}+\frac{\phi(\tau)}{2\epsilon}\right)}&0\\
0&e^{i\left(\frac{E\tau}{\epsilon^{\alpha+1}}+\frac{\phi(\tau)}{2\epsilon}\right)}\\
\end{pmatrix}$. Then
$X_{\epsilon}(\delta,\tau)$ satisfies 
\[\frac{dX_{\epsilon}(\delta,\tau)}{d\tau}=\left(\frac{1}{\epsilon}A^{\delta}(\tau)+\delta B_{\epsilon}(\tau)\right) X_{\epsilon}(\delta,\tau), \quad X_{\epsilon}(\delta,0)=e_1
 ,\]
 where $A^{\delta}(\tau)=-i\begin{pmatrix}
-\phi'(\tau)/2&\delta v(\tau)\\
\delta v(\tau)&\phi'(\tau)/2\\
\end{pmatrix} $ and $B_{\epsilon}(\tau)=
 \frac{-i}{\epsilon} 
\begin{pmatrix}
0&v(\tau)e^{i(\frac{4E\tau}{\epsilon^{\alpha+1}}+\frac{2 \phi(\tau)}{\epsilon})}\\
v(\tau)e^{-i(\frac{4E\tau}{\epsilon^{\alpha+1}} +\frac{2 \phi(\tau)}{\epsilon})}&0\\
\end{pmatrix}$.
By our choice of $v$ and $\phi$, $A^{\delta}(\tau)$ is $C^{\infty}$ w.r.t $(\tau,\delta)$ and satisfies (\ref{ugap}).
Applying Theorem \ref{UDrive}, we get that $\|\psi_{\epsilon}(\delta,1)-e^{i\theta_{\delta,\epsilon}}e_2\|\leq C \epsilon^{\text{min}(1,\alpha-1)}$ where $C>0$ is independent of $\delta,\epsilon$ and $\theta_{\delta,\epsilon}\in \R$.
The result follows.
\end{proof}
\end{teo}

\begin{figure}
\begin{center}
\includegraphics[scale=0.4]{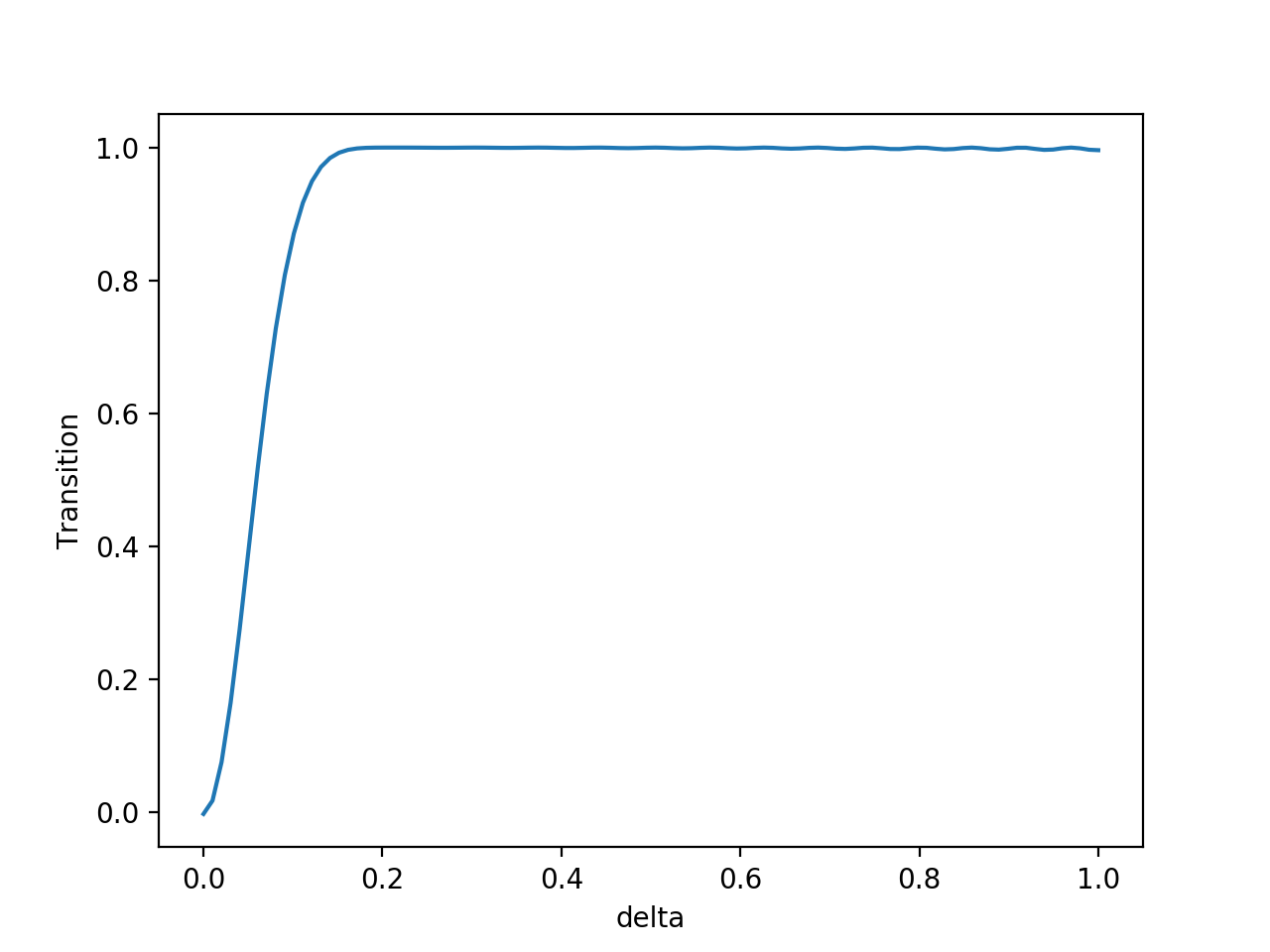}
\caption{Fidelity $|\langle \psi_{\epsilon}(\delta,1) , e_2
\rangle|^2$  as a function of the amplitude inhomogeneity $\delta$, with $\epsilon=0.01$, $\alpha=1.5$, and $E=1$.}
\label{inhomampli}
\end{center}
\end{figure}

\begin{figure}
\begin{center}
\includegraphics[scale=0.4]{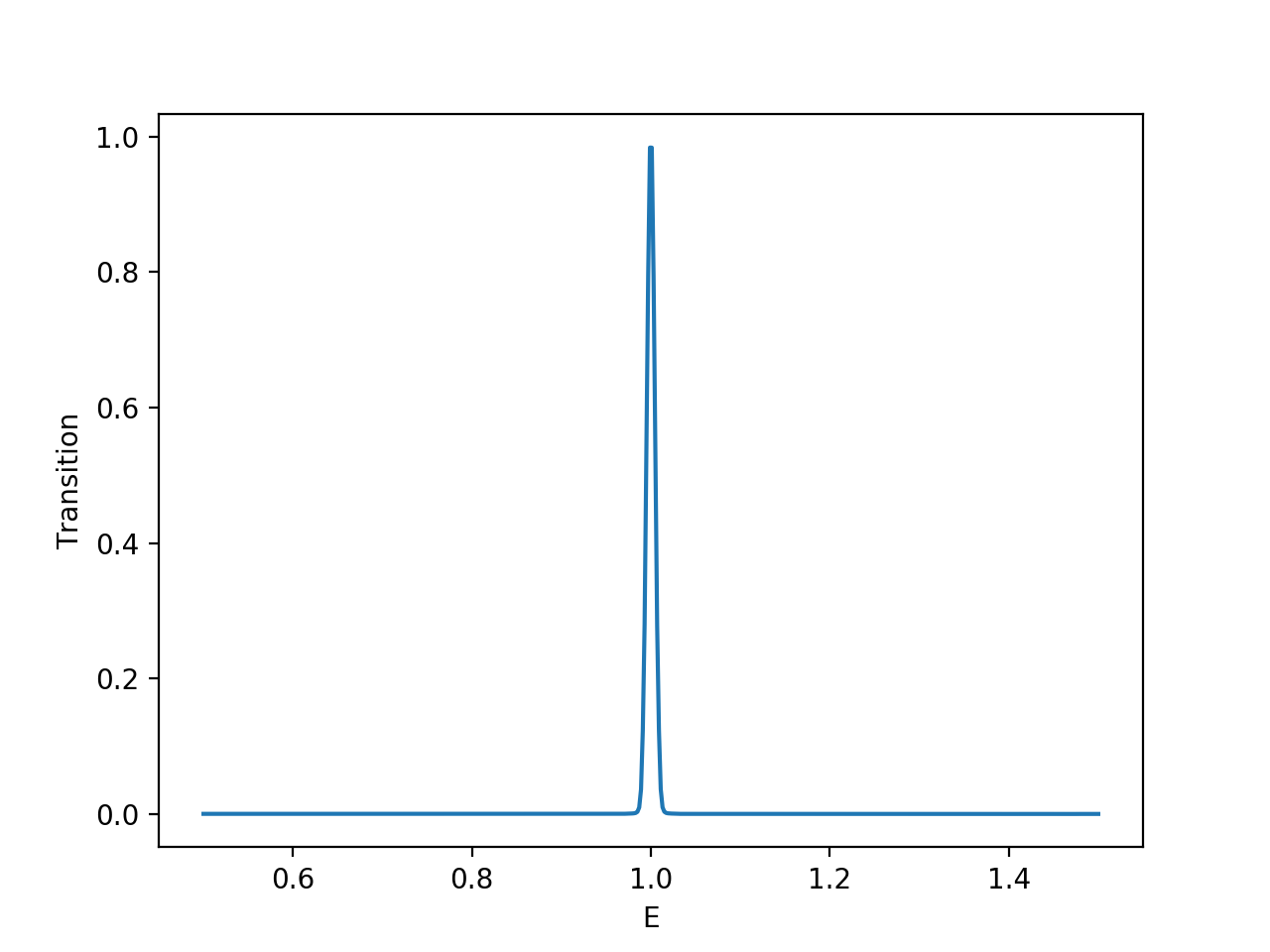}
\caption{Fidelity $|\langle \psi_{\epsilon}(E,1) , e_2
\rangle|^2$  as a function of the drift term  $E$, with $\epsilon=0.01$ and $\alpha=1.5$.}
\label{larmor}
\end{center}
\end{figure}
Consider the same $(v,\phi)$ as those chosen in Section \ref{control}. 
 For each $\delta\in [0,1]$,
let $\psi_{\epsilon}(\delta,\tau)$ be the solution of 
(\ref{robust}) with initial condition $\psi(\delta,0)=e_1$ and $E=1$.
We have plotted on Figure~\ref{inhomampli} the \emph{fidelity}, that is $|\langle \psi_{\epsilon}(\delta,1) , e_2
\rangle|^2$ for a dispersion $\delta$  of the amplitude of the control in $[0,1]$. On every sub-interval $[a,b]$ of $[0,1]$ with $a>0$, the fidelity converges uniformly to the constant function $\delta \mapsto 1$ when $\epsilon\to 0$.

Let now $\psi_{\epsilon}(E,\tau)$ be the solution of the equation
\begin{equation}
i\frac{d\psi_{\epsilon}(E,\tau)}{d\tau}=\begin{pmatrix}
\frac{E}{\epsilon^{\alpha+1}}& u_{\epsilon}(\tau)\\
 u_{\epsilon}(\tau)&-\frac{E}{\epsilon^{\alpha+1}}
\end{pmatrix}
\psi_{\epsilon}(E,\tau)
\end{equation}
where
$u_{\epsilon}(\tau)=\frac{2}{\epsilon}v(\tau)\cos(\frac{2\tau}{\epsilon^{\alpha+1}}+\frac{1}{\epsilon}\phi(\tau))$,
with initial condition $\psi(E,0)=e_1$ for every $E\in [\frac{1}{2},\frac{3}{2}]$.
We have plotted on Figure~\ref{larmor} the fidelity for a dispersion of $E$ in $[\frac{1}{2},\frac{3}{2}]$.
As already mentioned in the introduction, numerical simulations suggest that our method of control is not robust w.r.t. inhomogeneities of the resonance frequency $E$.

\bibliographystyle{abbrv}
\bibliography{bibli}

\end{document}